\documentclass[11pt]{amsart}
\usepackage[margin=1.1in]{geometry}
\usepackage{amsmath,amssymb,amsthm,mathtools}
\usepackage{booktabs,enumitem}
\usepackage[hidelinks]{hyperref}
\hypersetup{pdftitle={The rank of 3x3 matrix multiplication over F2 is 23},
  pdfauthor={Tejasvi Singh Tomar}}

\newtheorem{theorem}{Theorem}[section]
\newtheorem{lemma}[theorem]{Lemma}
\newtheorem{proposition}[theorem]{Proposition}
\newtheorem{corollary}[theorem]{Corollary}
\theoremstyle{definition}
\newtheorem{definition}[theorem]{Definition}
\theoremstyle{remark}

\numberwithin{equation}{section}
\newcommand{\F}{\mathbb{F}_2}
\newcommand{\Z}{\mathbb{Z}}
\newcommand{\GL}{\mathrm{GL}_3(\F)}
\newcommand{\mm}{\langle 3,3,3\rangle}
\newcommand{\rk}{\operatorname{rank}}
\newcommand{\lb}{\operatorname{lb}}
\newcommand{\cp}{\operatorname{cap}}

\newcommand{\cK}{\mathcal{K}}

\begin{document}

\title{The rank of $3\times 3$ matrix multiplication over $\F$ is $23$}
\author{Tejasvi Singh Tomar}
\renewcommand{\urladdrname}{{\itshape ORCID}}
\urladdr{\href{https://orcid.org/0000-0003-4668-3639}{https://orcid.org/0000-0003-4668-3639}}

\begin{abstract}
The rank of the tensor of $3\times 3$ matrix multiplication over the field with two elements is at most $23$ by Laderman's algorithm, and Rudich and Rousseau recently proved that it is at least $22$. We prove that it equals $23$. Hence Laderman's algorithm uses the fewest multiplications among all bilinear algorithms over $\F$ and among all bilinear algorithms with integer coefficients. The proof uses the substitution method in the form developed by D'Ambrosio, Wang and Yang et al., which bounds the number of first factors of a decomposition that lie in a given subspace of the space of first factors. We sharpen these bounds for $111$ of Wang's $496$ classes of subspaces, constrain the ranks of the first factors through a $27\times 27$ flattening of the tensor, and list, up to symmetry, all sets of $22$ first factors that satisfy these constraints. An exact completion search shows that none of them is the set of first factors of a decomposition. The computation emits certificates, which are checked in the Lean~4 proof assistant. In one formal proof the largest checks are evaluated as compiled code, and in a second one the kernel evaluates every check, so that its final theorem depends only on the three standard axioms of Lean. The Lean proofs build on the formalization of Rudich and Rousseau and import their table of bounds and their reduction for block algorithms.
\end{abstract}

\maketitle

\section{Introduction}\label{sec:intro}

Let $\mm$ be the tensor of $3\times 3$ matrix multiplication. Its rank $R_K(\mm)$ over a field $K$ is the smallest number of multiplications in a bilinear algorithm over $K$ that computes the product of two $3\times 3$ matrices \cite[Ch.~14]{BCS97}. Laderman's algorithm \cite{Lad76} uses $23$ multiplications. It does not use commutativity and its coefficients are $0$ and $\pm1$, so it works over any ring, and $R_K(\mm)\le 23$ for all fields $K$. For $2\times 2$ matrices the question was settled half a century ago: Strassen's algorithm uses $7$ multiplications \cite{Str69}, and Hopcroft and Kerr \cite{HK71} and Winograd \cite{Win71} proved that $7$ are necessary. For $3\times 3$ matrices, however, the best lower bound valid over all fields is $19$, due to Bläser \cite{Bla03}, and whether Laderman's $23$ multiplications are necessary has been open since 1976.

Over $\F$ the lower bound has risen several times in 2026. C.~Wang proved $R_{\F}(\mm)\ge 20$ with an automated method that splits the problem into $496$ classes of subspaces \cite{Wan26a}, and Beuchert checked his certificate independently \cite{Beu26}. Wang \cite{Wan26b} and, independently, Yang et al.\ \cite{Yan26} then proved $R_{\F}(\mm)\ge 21$, and Yang et al.\ formalized their proof in Lean. Tahir independently obtained $21$ conditionally on bounds from Wang's table \cite{Tah26}. Rudich and Rousseau proved $R_{\F}(\mm)\ge 22$, also in Lean \cite{RR26}. This left the question whether $22$ multiplications suffice over $\F$, and we show that they do not.

\begin{theorem}\label{thm:main}
$R_{\F}(\mm)=23$.
\end{theorem}

A bilinear algorithm with integer coefficients becomes a bilinear algorithm over $\F$ with the same number of multiplications when its coefficients are reduced modulo $2$. An algorithm with integer constants that can be applied to blocks of matrices cannot use commutativity, and by the argument of Hopcroft and Kerr \cite{HK71} a lower bound over $\F$ applies to it as well. Rudich and Rousseau explain this reduction and formalize it in Lean \cite{RR26}. Theorem~\ref{thm:main} therefore has the following consequence.

\begin{corollary}\label{cor:integer}
A bilinear algorithm over $\F$ or with integer coefficients that multiplies two $3\times 3$ matrices uses at least $23$ multiplications. In particular $R_{\Z}(\mm)=23$, and no recursive $3\times3$ algorithm with integer constants has an exponent below $\log_3 23\approx 2.854$.
\end{corollary}

Over fields of characteristic zero the rank of $\mm$ is still unknown. It lies between $19$ and $23$.

\subsection*{The method}
The proof is a computation organized by the substitution method \cite{Pan66,AFT11,LM17,Bla03,Bla13}. Let $A$ be the space of first factors. For a subspace $S\subseteq A$ let $R(S)$ be the rank of $\mm$ modulo $S$, that is, the rank of the problem in which the first input is restricted to the annihilator of $S$. In a decomposition of length $r$ at most $r-R(S)$ terms have their first factor in $S$ (Lemma~\ref{lem:cap}). For $r=22$ this bounds the number of first factors that each subspace can contain, the \emph{capacity} of the subspace. Wang bounded $R(S)$ for representatives of the $496$ orbits of subspaces under the symmetries of $\mm$ \cite{Wan26a}, and Rudich and Rousseau improved these bounds \cite{RR26}.

In Section~\ref{sec:raise} we raise the lower bound for $111$ classes of their table. For most of them we enumerate the possible first factors of a hypothetical shorter decomposition of the quotient and refute each candidate. The others follow from these by monotonicity. The most useful new bound is $R(S)\ge 21$ for a point $S$ spanned by a matrix of rank one. With the bounds of \cite{RR26} for the other two classes of points it shows that the $22$ first factors of a decomposition of length $22$ are distinct. By Lemma~\ref{lem:cap}, decompositions of length $23$ give upper bounds, and those of a collection of Kauers \cite{Kau} show that most of the new bounds are best possible.

Section~\ref{sec:types} constrains the numbers $n_1,n_2,n_3$ of first factors of rank $1$, $2$ and~$3$. These constraints come from a $27\times 27$ flattening of $\mm$ to which each term contributes the rank of its first factor, and hence $n_1+2n_2+3n_3\ge 27$. A Gram-matrix form of the identity behind this inequality excludes eight types $(n_1,n_2,n_3)$, three of them by short arguments and five by an exhaustive computation. Section~\ref{sec:l479} shows that, if at least $14$ first factors have rank one, no line in a certain orbit of $441$ lines contains two first factors, although the capacity of these lines is $2$. The proof assumes two first factors on such a line, enumerates the possible projections of the other $20$ terms and shows that none of them lifts back to a decomposition of $\mm$.

With these constraints we enumerate, up to the symmetry group of order $56{,}448$, all sets of $22$ nonzero matrices in $\F^{3\times 3}$ that satisfy them (Section~\ref{sec:search}). An exact search then tries to complete each set to a decomposition (Section~\ref{sec:tsolve}), and fails for every one of them.

\subsection*{Relation to earlier work}
Capacities, coset profiles, the symmetry-reduced enumeration of first factors and the deduction that the first factors are distinct come from the work of D'Ambrosio \cite{DAm26}, Wang \cite{Wan26a,Wan26b} and Yang et al.\ \cite{Yan26}, and the split flattening is due to Yang et al. In the proofs of $R_{\F}(\mm)\ge 21$ the conditions on the first factors alone exclude all cases. At $r=22$ they do not, and most of the work lies in refuting the sets of first factors that survive, by an exhaustive search for the second and third factors (Section~\ref{sec:tsolve}). Certificates that branch on these factors also occur in the Lean proof of \cite{RR26}. What is new here is this refutation step, applied to all $75{,}327{,}600$ sets that survive the search at $r=22$ (Table~\ref{tab:search}), the new lower bounds of Table~\ref{tab:raise}, the type exclusions of Proposition~\ref{prop:types} and Theorem~\ref{thm:l479}.

\subsection*{Verification}
Sections~\ref{sec:verify} and~\ref{sec:kernel} describe the verification. The searches emit certificates, which are verified by checkers written in Lean. The soundness of each checker is a theorem of Lean proved with its three standard axioms. There are two formal proofs. In the first, the checks for the search at $r=22$ and for the largest enumerations are evaluated as compiled code through \texttt{native\_decide}, which adds one axiom for each evaluation. This proof therefore depends on the Lean compiler and runtime as well as on the kernel. In the second (Section~\ref{sec:kernel}), the kernel evaluates every check, and its final theorem depends only on the axioms \texttt{propext}, \texttt{Classical.choice} and \texttt{Quot.sound}. Laderman's algorithm is also checked in Lean, which makes the equality $R_{\F}(\mm)=23$ a theorem there, and the bound of Corollary~\ref{cor:integer} for algorithms with integer constants that can be applied to blocks is derived in the formulation of \cite{RR26}. Statements marked with $\dagger$ depend on a computation.

\section{Preliminaries}\label{sec:prelim}

\subsection{The tensor}
Let $A=B=C=\F^{3\times 3}$ and let $e_{ij}$ be the matrix units. We write
\[
T=\mm=\sum_{i,j,k=1}^{3} e_{ij}\otimes e_{jk}\otimes e_{ki}\in A\otimes B\otimes C .
\]
A \emph{decomposition of length $r$} is an expression $T=\sum_{s=1}^{r}a_s\otimes b_s\otimes c_s$, and the rank of $T$ is the least such $r$. We call $a_s$ the \emph{first factors}. When $T$ has a decomposition of length $r$ it also has one of length $r+1$: replace a term $a\otimes b\otimes c$ by $a\otimes b\otimes c'+a\otimes b\otimes(c+c')$ with $c'\notin\{0,c\}$. In a decomposition of length $22$ all factors are nonzero, since otherwise $T$ would have a decomposition of length $21$, contrary to \cite{RR26}.

For a first factor $a$ we write $B$ and $C$ for the second and third factors read as matrices in the index order of $T$. Thus $b=\sum B_{jk}e_{jk}$ and $c=\sum C_{ki}e_{ki}$, and the triple $(a,b,c)$ is identified with $(a,B,C)$.

\subsection{Symmetries}
For $P,Q,R\in\GL$ the map
\begin{equation}\label{eq:sym}
(a,B,C)\mapsto (PaQ^{-1},\ QBR^{-1},\ RCP^{-1})
\end{equation}
fixes $T$, because $Q^{-1}$ and $Q$ cancel on the index $j$ shared by the first two factors, and $R$ and $P$ cancel in the same way on $k$ and $i$. So does $(a,B,C)\mapsto(a^{\mathsf T},C^{\mathsf T},B^{\mathsf T})$, which permutes the terms of $T$. Both maps send decompositions to decompositions of the same length. On first factors they generate the group $G$ of maps $a\mapsto PaQ$ and $a\mapsto Pa^{\mathsf T}Q$ with $P,Q\in\GL$. It has order $2\cdot 168^2=56{,}448$ and preserves the rank of a matrix. The space $A$ contains $511$ nonzero matrices, of which $49$ have rank one, $294$ rank two and $168$ rank three. The group $G$ acts transitively on each of these three sets.

\subsection{Quotients and capacities}\label{subsec:quot}
For a subspace $S\subseteq A$ let $\pi_S\colon A\to A/S$ be the projection and $T_S=(\pi_S\otimes 1\otimes 1)(T)$. We write $R(S)$ for the rank of $T_S$. If $S\subseteq S'$ then $R(S')\le R(S)$, and $R(gS)=R(S)$ for $g\in G$.

\begin{lemma}[{\cite[\S4]{DAm26}}, {\cite[Lemma~2.1]{Wan26b}}, {\cite[Lemma~2.2]{Yan26}}]\label{lem:cap}
If $T=\sum_{s=1}^{r}a_s\otimes b_s\otimes c_s$, then for every subspace $S\subseteq A$
\[
\#\{s : a_s\in S\}\le r-R(S).
\]
\end{lemma}

\begin{proof}
Applying $\pi_S\otimes 1\otimes 1$ removes the terms with $a_s\in S$ and leaves a decomposition of $T_S$ of length $\#\{s:a_s\notin S\}$.
\end{proof}

We call a subspace of dimension $1$ a \emph{point} and one of dimension $2$ a \emph{line}. Wang sorted the subspaces of $A$ into $496$ classes up to the action of $G$ \cite{Wan26a}. A \emph{table} assigns to each class a number $\lb(S)\le R(S)$. At $r=22$ the \emph{capacity} of $S$ is $\cp(S)=22-\lb(S)$, and Lemma~\ref{lem:cap} says that a subspace contains at most $\cp(S)$ first factors of a decomposition of length $22$. We start from the table of Rudich and Rousseau, which gives $R(S)$ exactly for $359$ classes and lower bounds for the others \cite{RR26}. In this table $R(\{0\})\ge 22$, and $R(S)\ge 21$ for the points spanned by matrices of rank two and three. In our data the classes are numbered as in the Lean table of \cite{RR26}, which lists a basis of a representative of each class.

The case $S=\ker\varphi$ of a hyperplane will be used often. Here $\varphi\colon A\to\F$ is a nonzero linear form, which we write as $\varphi(a)=\sum_{i,j}x_{ij}a_{ij}$ with a matrix $x\ne 0$, and we set $\rk\varphi=\rk x$. Contracting $T$ with $\varphi$ gives the $9\times 9$ matrix $T(\varphi)=\sum_{i,j,k}x_{ij}\, e_{jk}\otimes e_{ki}$. It consists of three copies of $x^{\mathsf T}$, one for each value of $k$, and so has rank $3\rk\varphi$. Hence $R(\ker\varphi)=3\rk\varphi$, and Lemma~\ref{lem:cap} becomes the statement that at least $3\rk\varphi$ first factors satisfy $\varphi(a_s)=1$.

\section{Lower bounds for quotients}\label{sec:raise}

\subsection{Coset profiles}\label{subsec:profiles}
Let $S$ be a subspace and $\ell$ an integer. To prove $R(S)\ge\ell$ we assume that $T_S$ has a decomposition of length $\ell-1$, which we may take with all factors nonzero (Section~\ref{sec:prelim}). Its first factors are nonzero cosets $\bar a_1,\dots,\bar a_{\ell-1}$ of $S$. Let $n(x)$ be the number of $t$ with $\bar a_t=x$. Applying Lemma~\ref{lem:cap} to the decomposition of $T_S$ and a superspace $S'\supseteq S$ gives (compare \cite[Lemma~2.2]{Yan26})
\begin{equation}\label{eq:coset}
\sum_{0\ne x\in S'/S} n(x)\le \ell-1-R(S')\le \ell-1-\lb(S').
\end{equation}
We call a function $n$ with $\sum_x n(x)=\ell-1$ that satisfies \eqref{eq:coset} for all $S'$ a \emph{coset profile}. The coset profiles are the integer points of a polytope with one coordinate per nonzero coset. We enumerate them up to the stabilizer of $S$ in $G$, or for six classes up to a subgroup of it, by an orderly branch and bound in which linear-programming bounds prune the branches. We then refute each coset profile by the completion search of Section~\ref{sec:tsolve} with target $T_S$. If all coset profiles are refuted, then $R(S)\ge\ell$. The bounds $\lb(S')$ used in \eqref{eq:coset} belong to strict superspaces of $S$, so we establish new bounds in decreasing order of $\dim S$.

\subsection{Points}
\begin{theorem}[$\dagger$]\label{thm:point}
If $S$ is spanned by a matrix of rank one, then $R(S)\ge 21$.
\end{theorem}

Here $A/S$ has $255$ nonzero cosets and the stabilizer of $S$ in $G$ has order $1152$. The enumeration of coset profiles of size $20$ has $809{,}336{,}049$ nodes and $7{,}409{,}768$ leaves, and all leaves are refuted.

\begin{corollary}[$\dagger$]\label{cor:distinct}
The $22$ first factors of a decomposition of length $22$ are pairwise distinct.
\end{corollary}

\begin{proof}
A repeated first factor $a$ would put two first factors in the point $\F a$. By Theorem~\ref{thm:point} and \cite{RR26} all points have $R\ge 21$ and therefore capacity at most $1$.
\end{proof}

\subsection{The other classes}
Table~\ref{tab:raise} lists the new bounds by dimension. We proved $88$ of them, among them Theorem~\ref{thm:point}, by the method of Section~\ref{subsec:profiles}. The other $23$ follow by monotonicity: for each of them some $gS$ with $g\in G$ is contained in a subspace $S'$ of larger dimension whose new bound is at least as large, and $R(S)=R(gS)\ge R(S')$. These are the bound $19\to 20$ for one class of dimension $2$, the bound $18\to 19$ for eight classes of dimension $3$, the bound $16\to 17$ for three classes of dimension $4$, and the bound $17\to 18$ for eleven classes of dimension $4$. For $106$ of the $111$ classes the new bound equals the upper bound of \cite{RR26}.

\begin{proposition}[$\dagger$]\label{prop:raise}
The lower bounds of Table~\ref{tab:raise} hold.
\end{proposition}

\begin{table}[ht]
\caption{New lower bounds on $R(S)$, by $\dim S$. A class is listed under the bound of \cite{RR26} and the new bound. The number in parentheses in the second column counts the classes whose new bound follows by monotonicity. The last two columns give the size of the enumerations of Section~\ref{subsec:profiles}.}\label{tab:raise}
\centering
\small
\begin{tabular}{@{}ccp{0.36\textwidth}rr@{}}
\toprule
$\dim S$ & classes & bounds & enumeration nodes & leaves\\
\midrule
$1$ & $1$ & $20\to 21$ & $809{,}336{,}049$ & $7{,}409{,}768$\\
$2$ & $4$ ($1$) & $19\to 20$ (1), $20\to 21$ (3) & $26{,}301{,}783{,}827$ & $5{,}211{,}895{,}431$\\
$3$ & $35$ ($8$) & $18\to 19$ (17), $19\to 20$ (18) & $3{,}654{,}254{,}117$ & $969{,}245{,}867$\\
$4$ & $50$ ($14$) & $16\to 17$ (8), $17\to 18$ (23), $18\to 19$ (19) & $46{,}041{,}278$ & $15{,}998{,}835$\\
$5$ & $21$ & $16\to 17$ (5), $16\to 18$ (1), $17\to 18$ (15) & $1{,}089{,}004$ & $349{,}120$\\
\bottomrule
\end{tabular}
\end{table}

\subsection{Upper bounds}\label{subsec:upper}
Lemma~\ref{lem:cap} also gives upper bounds: if a decomposition of length $23$ has $k$ first factors in $S$, then $R(S)\le 23-k$. We apply this to the collection of schemes of length $23$ that Kauers made available \cite{Kau}. It consists of the $64{,}150$ schemes in the data set of \cite[\S6]{KM22}, which extends that of \cite{HKS21}, and nine named schemes. Of its $64{,}159$ files, $62{,}239$ are decompositions over $\F$,\footnote{Each of the other $1{,}920$ files becomes a decomposition when one factor of one product is changed. These changes are listed in our data.} and we use each of them three times, once with each of its three factors in the role of the first factor. For $406$ of the $496$ classes, among them $103$ of the $111$ classes of Table~\ref{tab:raise}, one of these decompositions has $23-\lb(S)$ first factors in a subspace of the class, and hence $R(S)=\lb(S)$. For classes $293$ and $488$ in the numbering of \cite{RR26}, which keep the bounds of \cite{RR26}, the upper bounds obtained in this way are new, and they give $R(S)=17$ and $R(S)=20$, respectively. For class $473$ of Table~\ref{tab:raise} the new upper bound is $20$, which equals its new lower bound. Hence $R(S)$ is now known exactly for $107$ of the $111$ classes of Table~\ref{tab:raise}, namely class $473$ and the $106$ whose new bound equals the upper bound of \cite{RR26}. For class $468$ the collection gives $R(S)\le 20$, against the lower bound $19$. For example, one decomposition in the collection has the three matrices $I$, $e_{12}+e_{13}+e_{31}$ and their sum among its first factors. They are the three nonzero elements of a line of class $488$ and lie in planes of classes $468$ and $473$. By Lemma~\ref{lem:cap} this gives the bound $20$ for these three classes. For every class, our data list a decomposition with the largest number of first factors in a subspace of the class.

\section{The ranks of the first factors}\label{sec:types}

In this section we consider a decomposition $T=\sum_{t=1}^{r}a_t\otimes b_t\otimes c_t$ with all factors nonzero. Its \emph{type} is $(n_1,n_2,n_3)$, where $n_q$ is the number of first factors of rank $q$. We write $r_t=\rk(a_t)$ and $\Delta=\sum_t r_t-27$.

\subsection{The split flattening}
Index the rows of a $27\times 27$ matrix by pairs $(i,\beta)$ and the columns by pairs $(j,\gamma)$, where $i,j\in\{1,2,3\}$ and $\beta,\gamma$ run over the nine matrix positions. For $a\in A$, $b\in B$, $c\in C$ let $M(a,b,c)$ be the matrix with entries $a_{ij}b_\beta c_\gamma$. Its rank is $\rk(a)$ when $b,c\ne 0$, since $M(a,b,c)$ is the Kronecker product of $a$ and $bc^{\mathsf T}$ up to the order of indices. Let $P=\sum_t M(a_t,b_t,c_t)$. By linearity $P$ depends only on $T$, and its only nonzero entries are $P_{(i,jk),(j,ki)}=1$. Thus $P$ is a permutation matrix.

\begin{lemma}[{\cite[(4.2)]{Yan26}}]\label{lem:split}
$\sum_t\rk(a_t)\ge 27$, that is, $\Delta\ge 0$.
\end{lemma}

\begin{proof}
$27=\rk P\le\sum_t\rk M(a_t,b_t,c_t)=\sum_t\rk(a_t)$.
\end{proof}

For a decomposition of length $22$ we have $\Delta=n_2+2n_3-5$, so Lemma~\ref{lem:split} gives $n_2+2n_3\ge 5$.

\subsection{The Gram form}
The inequality of Lemma~\ref{lem:split} does not use how the terms fit together. The following factorization of $P$ records this information as equations between pairs of terms. Factor each first factor as $a_t=L_tR_t^{\mathsf T}$ with $L_t,R_t\in\F^{3\times r_t}$. Let $X$ be the $27\times(27+\Delta)$ matrix whose column blocks are $L_t\otimes b_t$, with rows indexed like the rows of $P$, and let $Y$ be the $27\times(27+\Delta)$ matrix with column blocks $R_t\otimes c_t$, with rows indexed like the columns of $P$. Then $M(a_t,b_t,c_t)=(L_t\otimes b_t)(R_t\otimes c_t)^{\mathsf T}$, so $P=XY^{\mathsf T}$.

\begin{lemma}\label{lem:gram}
The matrix $\Gamma=Y^{\mathsf T}P^{-1}X$ is idempotent of rank $27$, and $\Pi=I+\Gamma$ is idempotent of rank $\Delta$. Its blocks are
\begin{equation}\label{eq:gram}
\Gamma_{ts}=R_t^{\mathsf T}B_sC_tL_s .
\end{equation}
\end{lemma}

\begin{proof}
$\Gamma^2=Y^{\mathsf T}P^{-1}(XY^{\mathsf T})P^{-1}X=\Gamma$. The rank of $\Gamma$ is at most $27$, and $X\Gamma Y^{\mathsf T}=P$ has rank $27$, so $\rk\Gamma=27$. Over $\F$ the matrix $I+\Gamma=I-\Gamma$ is the complementary idempotent, of rank $(27+\Delta)-27$. For \eqref{eq:gram}, $P^{-1}=P^{\mathsf T}$ and $P^{\mathsf T}_{(j,\gamma),(i,\beta)}=1$ exactly when $\beta=jk$ and $\gamma=ki$ for some $k$. Hence the $(\alpha,\alpha')$ entry of $\Gamma_{ts}$ is $\sum_{i,j,k}(R_t)_{j\alpha}(B_s)_{jk}(C_t)_{ki}(L_s)_{i\alpha'}$.
\end{proof}

Write $\Pi=UV^{\mathsf T}$ with $U,V$ of size $(27+\Delta)\times\Delta$ and let $U_t,V_t$ be their row blocks. Multiplying \eqref{eq:gram} by $L_t$ on the left and by $R_s^{\mathsf T}$ on the right, and setting $l_t=L_tU_t$ and $m_s=R_sV_s$, we obtain
\begin{equation}\label{eq:excess}
a_tB_sC_ta_s=\delta_{ts}\,a_t+l_tm_s^{\mathsf T}\qquad\text{for all }t,s,
\end{equation}
with $l_t,m_t\in\F^{3\times\Delta}$ whose columns lie in the column spaces of $a_t$ and $a_t^{\mathsf T}$. For $\Delta=0$ this is the rank-additive saturation lemma of Yang et al.\ \cite[Lemma~5.1]{Yan26}. Its solutions are invariant under the symmetries \eqref{eq:sym}, and under the change $(l_t,m_t)\mapsto(l_tg,m_tg^{-\mathsf T})$ with $g\in\mathrm{GL}_\Delta(\F)$.

\subsection{Excess zero}

\begin{lemma}\label{lem:S}
If $\Delta=0$ and some first factor is invertible, then $r=25$. In particular, a decomposition of length $22$ with $\Delta=0$ has type $(17,5,0)$.
\end{lemma}

\begin{proof}
For $\Delta=0$, \eqref{eq:excess} reads $a_tB_sC_ta_s=\delta_{ts}a_t$. Let $a_0$ be invertible. Then $a_0B_0C_0a_0=a_0$ gives $B_0C_0=a_0^{-1}$, so $C_0$ is invertible. For $s\ne 0$ the equation $a_0B_sC_0a_s=0$ shows that $B_s$ vanishes on $C_0\,\mathrm{col}(a_s)$, a space of dimension $r_s$. Hence $\rk B_s\le 3-r_s$. The equation $a_sB_sC_sa_s=a_s$ gives $r_s\le\rk B_s$. Therefore $r_s\le 3-r_s$, that is $r_s=1$, for all $s\ne 0$. Then $27=\sum_tr_t=3+(r-1)$. For $r=22$ there is therefore no invertible first factor, and $n_1+n_2=22$, $n_1+2n_2=27$ give $(17,5,0)$.
\end{proof}

Lemma~\ref{lem:S} sharpens \cite[Proposition~5.3 and Corollary~6.1]{Yan26}, by which a decomposition with $\sum_tr_t=27$ has at most one invertible first factor. The same argument for $\langle 2,2,2\rangle$ shows that a decomposition with $\sum_t r_t=8$ and an invertible first factor has length $7$ and all other first factors of rank one, as in Strassen's algorithm.

\subsection{Three invertible first factors}

\begin{lemma}\label{lem:S3}
A decomposition with three distinct invertible first factors has $\Delta\ge 2$.
\end{lemma}

\begin{proof}
By Lemma~\ref{lem:S} we may assume $\Delta=1$, so that $\Pi=pq^{\mathsf T}$ with $q^{\mathsf T}p=1$. Let $x,y$ range over the indices of the three invertible terms, take $L_x=a_x$ and $R_x=I$, and let $p_x,q_x\in\F^3$ be the blocks of $p$ and $q$. Since $\Gamma=I+\Pi$, \eqref{eq:gram} gives
\[
Z_{yx}:=B_yC_x=(\delta_{xy}I+p_xq_y^{\mathsf T})\,a_y^{-1}.
\]
The $9\times 9$ matrix $[Z_{yx}]_{y,x}$ is the product of the block column $[B_y]_y$ and the block row $[C_x]_x$, so its rank is at most $3$.

(i) Each $Z_{xx}$ is singular. Otherwise, for $y\ne x$, the $6\times 6$ submatrix on the blocks $x,y$ has rank $3+\rk(Z_{yy}-Z_{yx}Z_{xx}^{-1}Z_{xy})\ge 3+(2-1)$, because $Z_{yy}$ has rank at least $2$ and the correction has rank at most $1$. This contradicts the bound $3$. Hence $\det(I+p_xq_x^{\mathsf T})=1+q_x^{\mathsf T}p_x=0$, so $q_x^{\mathsf T}p_x=1$, $\rk Z_{xx}=2$ and $\mathrm{im}\,Z_{xx}=q_x^{\perp}$.

(ii) $\rk B_x=\rk C_x=2$. Both ranks are at least $\rk Z_{xx}=2$. If $C_x$ were invertible, then $Z_{yx}=B_yC_x$ would have rank at least $2$ for $y\ne x$, but $Z_{yx}=p_xq_y^{\mathsf T}a_y^{-1}$ has rank at most $1$. The same argument applies to $B_y$.

(iii) $q_y^{\mathsf T}p_x=\delta_{xy}$. By (i) and (ii), $\mathrm{im}\,B_y=\mathrm{im}\,Z_{yy}=q_y^{\perp}$. For $x\ne y$ the matrix $Z_{yx}=p_xq_y^{\mathsf T}a_y^{-1}$ is nonzero and its image $\F p_x$ lies in $\mathrm{im}\,B_y$, so $q_y^{\mathsf T}p_x=0$.

(iv) Let $V$ be the row space of $[C_x]_x$, of dimension at most $3$. Each block row $B_y[C_x]_x$ has its rows in $V$. For $z\ne y$ the row $q_z^{\mathsf T}B_y[C_x]_x$ has $q_z^{\mathsf T}a_y^{-1}$ in block $y$, $q_y^{\mathsf T}a_y^{-1}$ in block $z$, and $0$ in the third block. Call it $v_{y,z}$.

(v) Fix a block $x_0$ and let $y,z$ be the other two indices. The projection of $V$ to block $x_0$ contains $q_z^{\mathsf T}a_{x_0}^{-1}$ for both $z\ne x_0$, and these are independent. Hence the vectors of $V$ that vanish on block $x_0$ form a space of dimension at most $1$. It contains the nonzero vectors $v_{y,z}$ and $v_{z,y}$, which are therefore equal. This gives $q_z^{\mathsf T}a_y^{-1}=q_z^{\mathsf T}a_z^{-1}$ and $q_y^{\mathsf T}a_y^{-1}=q_y^{\mathsf T}a_z^{-1}$.

(vi) Over the three choices of $x_0$ we get $q_x^{\mathsf T}a_t^{-1}=q_x^{\mathsf T}a_s^{-1}=q_x^{\mathsf T}a_w^{-1}$ for all $x$, where $t,s,w$ are the three indices. The $q_x$ form a basis by (iii), so the three first factors are equal, a contradiction.
\end{proof}

\subsection{An exhaustive computation}
For larger configurations we solve \eqref{eq:excess} by computer. Fix a type and a value of $\Delta$. Using \eqref{eq:sym} and the action of $\mathrm{GL}_\Delta(\F)$ we bring one term, the \emph{anchor}, to $a_0=I$ (or to a fixed matrix of rank two when there is no invertible term) and its data $(B_0,C_0,l_0,m_0)$ to one of finitely many orbit representatives: $7$ for $\Delta=1$ and $36$ for $\Delta=2$ with an invertible anchor, and $4$ for $\Delta=0$ with an anchor of rank two. The other terms become \emph{vertices}: tuples $(a,B,C,l,m)$ that satisfy \eqref{eq:excess} with themselves and with the anchor. For fixed $a$ these conditions are linear in $B$ once $m$ is fixed and linear in $C$ once $B$ and $l$ are fixed, so the vertices are found by solving linear systems. Two vertices are adjacent when their first factors differ and \eqref{eq:excess} holds for the pair. A decomposition of the given type yields a clique with the prescribed number of vertices of each rank. We first list the cliques formed by the anchor's companions of rank two and three (the \emph{cores}). For each core we then bound the largest set of pairwise adjacent rank-one vertices compatible with the core and having distinct first factors, using a colouring of the compatibility graph as the certificate.

\begin{proposition}[$\dagger$]\label{prop:types}
No decomposition of length $22$ has type $(17,5,0)$, $(18,2,2)$, $(17,4,1)$, $(17,3,2)$ or $(18,1,3)$.
\end{proposition}

\begin{table}[ht]
\caption{The computation behind Proposition~\ref{prop:types}. The last column is the number of rank-one terms by which a core would have to be extended, and no core extends that far.}\label{tab:types}
\centering
\small
\begin{tabular}{@{}cccrc@{}}
\toprule
type & $\Delta$ & anchor orbits & cores & needed\\
\midrule
$(17,5,0)$ & $0$ & $4$ & $23{,}175$ & $17$\\
$(18,2,2)$ & $1$ & $7$ & $0$ & $18$\\
$(17,4,1)$ & $1$ & $7$ & $708$ & $17$\\
$(17,3,2)$ & $2$ & $36$ & $96{,}717$ & $17$\\
$(18,1,3)$ & $2$ & $36$ & $68$ & $18$\\
\bottomrule
\end{tabular}
\end{table}

Lemmas~\ref{lem:S} and~\ref{lem:S3} exclude $(18,3,1)$, $(19,1,2)$ and $(19,0,3)$. Since $\Delta=0$ forces one of the types $(17,5,0)$, $(18,3,1)$, $(19,1,2)$, we obtain the following.

\begin{corollary}[$\dagger$]\label{cor:types}
Every decomposition of length $22$ satisfies $\sum_t\rk(a_t)\ge 28$, and none has type $(19,0,3)$, $(18,2,2)$, $(17,4,1)$, $(17,3,2)$ or $(18,1,3)$.
\end{corollary}

\section{Two first factors on a line}\label{sec:l479}

Let $L_0$ be the line spanned by $e_{11}$ and $e_{12}+e_{21}$. Its orbit under $G$ consists of $441$ lines, each containing one matrix of rank one and two of rank two. By Table~\ref{tab:raise} we have $R(L_0)\ge 20$, and Rudich and Rousseau prove $R(L_0)\le 20$ in Lean \cite{RR26}. The capacity of $L_0$ is therefore $2$, so Lemma~\ref{lem:cap} allows two first factors of a decomposition of length $22$ in a line of this orbit. We show that this does not happen when at least $14$ first factors have rank one, which is case (II) of the search in Section~\ref{sec:search}.

\begin{theorem}[$\dagger$]\label{thm:l479}
In a decomposition of length $22$ with at least $14$ first factors of rank one, each line in the $G$-orbit of $L_0$ contains at most one first factor.
\end{theorem}

The proof uses a criterion for lifting a decomposition of a quotient $T_L$ to a decomposition of $T$. Let $L$ be a line with nonzero elements $p$, $q$ and $p+q$, and choose linear forms $\chi_0,\chi_1$ on $A$ with $\chi_0(p)=1$, $\chi_0(q)=0$, $\chi_1(p)=0$ and $\chi_1(q)=1$. Let
\[
T_L=\sum_{t=1}^{20}\bar a_t\otimes b_t\otimes c_t
\]
be a decomposition with nonzero factors. For a linear form $\varphi$ on $A$ put
\[
V_\varphi=\Bigl\{\,T(\varphi)-\sum_{t=1}^{20}\varphi(a_t)\,b_tc_t^{\mathsf T}\ :\ a_t\in\bar a_t\text{ for all }t\,\Bigr\},
\]
a set of $9\times 9$ matrices, where $T(\varphi)$ is the contraction of Section~\ref{sec:prelim}.

\begin{lemma}[Lifting criterion]\label{lem:lift}
The following are equivalent.
\begin{enumerate}[label=(\alph*)]
\item There are $a_t\in\bar a_t$ and two further terms with distinct first factors in $L$ that together with the terms $a_t\otimes b_t\otimes c_t$ form a decomposition of $T$ of length $22$.
\item At least two of the sets $V_{\chi_0}$, $V_{\chi_1}$ and $V_{\chi_0+\chi_1}$ contain a matrix of rank one.
\end{enumerate}
\end{lemma}

\begin{proof}
Fix representatives $a_t^0\in\bar a_t$ and write $a_t=a_t^0+e_t$ with $e_t\in L$. As $e_t$ runs through $L$, the pair $(\chi_0(e_t),\chi_1(e_t))$ takes each value in $\F^2$ once. For a choice of the $e_t$ the tensor $T-\sum_ta_t\otimes b_t\otimes c_t$ maps to $0$ in $(A/L)\otimes B\otimes C$, so it equals $p\otimes M_p+q\otimes M_q$ for some matrices $M_p$ and $M_q$. Contracting with $\chi_0$, $\chi_1$ and $\chi_0+\chi_1$ shows that $M_p$, $M_q$ and $M_p+M_q$ are the elements of $V_{\chi_0}$, $V_{\chi_1}$ and $V_{\chi_0+\chi_1}$ given by this choice. Two terms with distinct first factors $x,y\in L$ and nonzero second and third factors have the form $x\otimes N_x+y\otimes N_y$ with $\rk N_x=\rk N_y=1$. Since
\[
p\otimes M_p+q\otimes M_q=p\otimes(M_p+M_q)+(p+q)\otimes M_q=q\otimes(M_p+M_q)+(p+q)\otimes M_p ,
\]
the two terms exist for this choice exactly when two of $M_p$, $M_q$ and $M_p+M_q$ have rank one. The element of $V_{\chi_0}$ depends only on the bits $\chi_0(e_t)$, that of $V_{\chi_1}$ only on the bits $\chi_1(e_t)$, and that of $V_{\chi_0+\chi_1}$ only on their sums. Any two of these three families of bits can be chosen independently. Hence two of the three matrices can be made of rank one by the same choice if and only if each of them can be made of rank one by some choice.
\end{proof}

For a single term whose slice has rank one, a related statement is part of \cite[Proposition~3.1]{LM17}, and Wang's forced-product technique \cite{Wan26a} also enumerates the free coefficients of a residual and tests its rank. To test whether $V_\varphi$, an affine space of matrices, contains a matrix $bc^{\mathsf T}$ of rank one, we solve for each of the $511$ nonzero vectors $b$ a linear system in $c$ and the coefficients.

\begin{proof}[Proof of Theorem~\ref{thm:l479}]
By symmetry it suffices to treat $L_0$. Suppose two first factors lie in $L_0$. They are distinct by Corollary~\ref{cor:distinct}, and the capacity of $L_0$ is $2$, so the third nonzero matrix of $L_0$ is not a first factor. The other $20$ terms project to a decomposition of $T_{L_0}$ of length $20$ with nonzero factors, and its coset profile $n$ satisfies \eqref{eq:coset} with $\ell-1=20$. At most one of the two first factors in $L_0$ has rank one, so at least $13$ of the other $20$ have rank one. Only $39$ of the $127$ nonzero cosets of $L_0$ contain a matrix of rank one, which gives the additional inequality
\[
\sum_{x\in C_1}n(x)\ge 13,
\]
where $C_1$ is the set of these $39$ cosets. We enumerate the coset profiles that satisfy these inequalities, as in Section~\ref{subsec:profiles}, and run the completion search with target $T_{L_0}$ on each. When it finds decompositions of $T_{L_0}$ with a given profile, we list all of them and check that for each at most one of the three sets in Lemma~\ref{lem:lift}(b) contains a matrix of rank one. By Lemma~\ref{lem:lift} no profile comes from a decomposition of $T$ of length $22$. The enumeration has $737{,}787{,}637{,}822$ nodes and $189{,}298{,}650{,}178$ leaves. For $10$ of the profiles the completion search finds decompositions of $T_{L_0}$, $540$ in all, and none of them lifts.
\end{proof}

The reduction to coset profiles works for any subspace in place of $L_0$ and any number of first factors in it, with the inequality for $C_1$ adjusted to the matrices of rank one in the subspace. The formalization proves it in this generality, together with the refutation of profiles that admit no decomposition of the quotient.

\section{The search at \texorpdfstring{$r=22$}{r=22}}\label{sec:search}

\subsection{Admissible sets}
By Corollary~\ref{cor:distinct} the first factors of a decomposition of length $22$ form a set $F\subset A\setminus\{0\}$ with $|F|=22$. Let $k$ be the number of elements of rank one in $F$. We split the proof into three cases: (I) $k\le 12$, (I$'$) $k=13$ and (II) $k\ge 14$, which is the range of Theorem~\ref{thm:l479}.

\begin{definition}\label{def:adm}
A set $F\subset A\setminus\{0\}$ with $|F|=22$ is \emph{admissible} if
\begin{enumerate}[label=(\alph*)]
\item $|F\cap S|\le\cp(S)$ for every subspace $S$ in the family $\cK$ described below, and
\item $|\{a\in F: \varphi(a)=1\}|\ge 3\rk\varphi$ for every nonzero linear form $\varphi$ on $A$,
\end{enumerate}
and, in case (II), in addition
\begin{enumerate}[label=(\alph*),resume]
\item the type of $F$ is not excluded by Corollary~\ref{cor:types}, and
\item each line in the $G$-orbit of $L_0$ (Section~\ref{sec:l479}) contains at most one element of $F$.
\end{enumerate}
\end{definition}

The family $\cK$ consists of the subspaces whose capacity is not implied by the capacity of another subspace, about $3.8\cdot 10^{6}$ of them. For a subspace $S$ of dimension at most $4$ the capacity can be lowered further. Since the elements of $F$ are distinct, $|F\cap S|\le k$ whenever no $k+1$ nonzero elements of $S$ respect the capacities of all subspaces of $S$. We use the bound obtained in this way for five classes of subspaces, for which it is smaller than $\cp(S)$. By Lemma~\ref{lem:cap}, Corollary~\ref{cor:types} and Theorem~\ref{thm:l479}, the set of first factors of a decomposition of length $22$ is admissible, and so are its images under $G$.

\subsection{Orderly generation}\label{subsec:orderly}
We order the $511$ nonzero matrices by rank, and matrices of equal rank by the integer $\sum_{i,j}a_{ij}2^{3i+j-4}$. In case (II) the order is rank one, rank two, rank three. In cases (I) and (I$'$) it is rank three, rank two, rank one, which keeps the search tree small when there are few elements of rank one. For finite sets $F\ne F'$ of the same size we write $F<F'$ if the smallest element of the symmetric difference belongs to $F$. A set is \emph{canonical} if it is the smallest set in its $G$-orbit. The search follows the design of Wang's profile enumerator \cite[\S3.3]{Wan26b}. Orderly generation goes back to Read \cite{Rea78} (see also \cite{McK98}).

\begin{lemma}\label{lem:canon}
Every $G$-orbit of sets contains a canonical set, and every initial segment of a canonical set is canonical.
\end{lemma}

\begin{proof}
The first claim is clear, since orbits are finite. For the second, let $F$ be canonical, let $c$ be a nonzero matrix and let $F_0=\{a\in F:a<c\}$. Suppose that $gF_0<F_0$ for some $g\in G$, and let $m$ be the smallest element of $gF_0\triangle F_0$. Then $m\in gF_0\setminus F_0$. If $m\ge c$, every element of $F_0$ is smaller than $m$ and therefore lies in $gF_0$, so $F_0=gF_0$ because both have the same size. This is impossible, so $m<c$. Every element of $F$ smaller than $m$ then lies in $F_0$, hence in $gF_0\subseteq gF$, while $m\in gF\setminus F$. Hence the smallest element of $gF\triangle F$ lies in $gF$, that is, $gF<F$. This contradicts the choice of $F$.
\end{proof}

The search builds sets in increasing order. A node is a canonical set $F_0$ of size less than $22$ that satisfies all constraints restricted to $F_0$. Its children are the sets $F_0\cup\{u\}$ with $u$ larger than all elements of $F_0$. By Lemma~\ref{lem:canon} the search reaches each canonical admissible set through its initial segments. It may therefore discard any child that has no canonical admissible completion.

\subsection{Discarding children}
A child $F_0\cup\{u\}$ is discarded if one of the following holds. Each condition shows that no admissible canonical set has $F_0\cup\{u\}$ as an initial segment.
\begin{enumerate}[label=(\roman*)]
\item $F_0\cup\{u\}$ violates a capacity in $\cK$.
\item For some nonzero $\varphi$, the elements of $F_0\cup\{u\}$ with $\varphi(a)=1$, together with all larger matrices with $\varphi(a)=1$ that can still be added, number fewer than $3\rk\varphi$.
\item The case condition on $k$ cannot be met by the matrices that can still be added.
\item In case (II), every type that a completion can have is excluded by Corollary~\ref{cor:types}.
\item Fewer than $21-|F_0|$ matrices larger than $u$ remain that can still be added.
\item No set of $21-|F_0|$ further matrices can be added without violating a capacity in $\cK$. The search proves this by a small branch and bound over the remaining candidates. When such a tree would be large, the child is kept as a node instead.
\item $F_0\cup\{u\}$ is not canonical, as witnessed by an element $g\in G$ with $g(F_0\cup\{u\})<F_0\cup\{u\}$.
\end{enumerate}
The same tests applied to a node itself may discard all of its children at once. The leaves of the search are the canonical admissible sets.

\subsection{Results}
Table~\ref{tab:search} gives the size of the search in each case.

\begin{table}[ht]
\caption{Size of the search at $r=22$.}\label{tab:search}
\centering
\begin{tabular}{@{}lrrr@{}}
\toprule
case & nodes & admissible sets & refuted\\
\midrule
(I) $k\le 12$ & $629{,}424$ & $3{,}580$ & all\\
(I$'$) $k=13$ & $19{,}547{,}416$ & $146{,}319$ & all\\
(II) $k\ge 14$ & $3{,}377{,}202{,}242$ & $75{,}177{,}701$ & all\\
\bottomrule
\end{tabular}
\end{table}

\begin{proposition}[$\dagger$]\label{prop:search}
Every canonical admissible set appears as a leaf of the search, and no leaf is the set of first factors of a decomposition of $T$.
\end{proposition}

\begin{proof}[Proof of Theorem~\ref{thm:main}]
Suppose $T$ has a decomposition of length $22$. By Corollary~\ref{cor:distinct} its first factors form a set $F$ of size $22$, admissible by the remark after Definition~\ref{def:adm}. By Lemma~\ref{lem:canon} some $gF$ with $g\in G$ is canonical, and $gF$ is the set of first factors of a decomposition because $g$ extends to a symmetry of $T$. This contradicts Proposition~\ref{prop:search}. Hence $R_{\F}(\mm)\ge 23$, and Laderman's algorithm gives equality.
\end{proof}

\section{The completion search}\label{sec:tsolve}

Given nonzero first factors $a_1,\dots,a_r$ and a target tensor $T'$ (either $T$ or a quotient $T_S$, written in suitable coordinates), the completion search decides whether there are nonzero $b_t,c_t$ with $\sum_ta_t\otimes b_t\otimes c_t=T'$. For a linear form $\varphi$ on the space of first factors let $T'(\varphi)$ be the contraction, a $9\times 9$ matrix.

The search assigns the pairs $(b_t,c_t)$ one term at a time. For a partial assignment and each nonzero $\varphi$ it keeps the residual
\[
\mathrm{Res}(\varphi)=T'(\varphi)-\sum_{t\text{ assigned},\,\varphi(a_t)=1}b_tc_t^{\mathsf T}
\]
and the number $u(\varphi)$ of unassigned terms with $\varphi(a_t)=1$. A completion exists only if $\rk\mathrm{Res}(\varphi)\le u(\varphi)$ for all $\varphi$, since the unassigned terms with $\varphi(a_t)=1$ decompose $\mathrm{Res}(\varphi)$. When all terms are assigned, the decomposition identity holds exactly when $\mathrm{Res}(\varphi)=0$ for all $\varphi$.

The candidates for the next term $t$ are restricted by necessary conditions only, so the search is exhaustive.
\begin{itemize}
\item If $\varphi(a_t)=1$ and $\varphi$ is tight, that is $\rk\mathrm{Res}(\varphi)=u(\varphi)$, then the unassigned terms of $\varphi$ give a decomposition of $\mathrm{Res}(\varphi)$ into $u(\varphi)$ matrices of rank one. Their vectors $b$ form a basis of the column space and their vectors $c$ a basis of the row space. Hence $b_t$ lies in the column space of $\mathrm{Res}(\varphi)$, and $c_t$ lies in the row space and satisfies $c_t^{\mathsf T}x=1$ for every $x$ with $\mathrm{Res}(\varphi)x=b_t$.
\item If $\varphi(a_t)=1$ and $\rk\mathrm{Res}(\varphi)=u(\varphi)-1$, then $b_t$ lies in the column space or $c_t$ lies in the row space of $\mathrm{Res}(\varphi)$.
\end{itemize}
These rules are dynamic versions of the restriction lemmas of \cite[Lemmas~2 and~3]{DAm26} and \cite[Lemma~4.1]{Wan26b}.
The maps $b\mapsto bS^{-1}$, $c\mapsto Sc$ with $S\in\GL$ fix all first factors, $T$ and the quotient targets we use. We may therefore restrict the first term to one representative $b$ of each orbit, and its third factor $c$ to one representative of each orbit of the stabilizer of $b$.

\section{Verification with compiled checks}\label{sec:verify}

\subsection{The formal statement}
The final theorem states that $n\ge 23$ whenever maps $a$, $b$ and $c$ from $\{1,\dots,n\}$ to $\F^{3\times 3}$ satisfy $\sum_s a_s\otimes b_s\otimes c_s=T$. The factors are arbitrary matrices, $\F$ is the field \texttt{ZMod 2} of Mathlib, and the definitions of the matrices, of $T$ and of a decomposition are the elementary ones, written out in a short file. Further theorems show that the statement does not change when $T$ is replaced by the tensor of the bilinear map $(A,B)\mapsto AB$, and that Laderman's algorithm reduced modulo $2$ is a decomposition of length $23$, so that $R_{\F}(\mm)=23$ is a theorem in Lean. Rudich and Rousseau formalized the reduction from algorithms with integer constants that can be applied to blocks \cite{RR26}. Through their formalization the final theorem gives that such an algorithm for $3\times 3$ matrices uses at least $23$ multiplications, which is their result with $23$ in place of $22$.

\subsection{Two layers}
The first layer consists of theorems proved in Lean~4 \cite{dMU21} with Mathlib \cite{mC20}, using only the axioms \texttt{propext}, \texttt{Classical.choice} and \texttt{Quot.sound}. For each kind of certificate it defines a checker, a function with Boolean values, and proves that the checker is sound: if it returns \texttt{true} on some data, the statement that the data certify holds. There are checkers for the nodes of the search at $r=22$ and the reasons for discarding their children, for canonicity, for the enumerations of coset profiles, for the refutation of profiles by the completion search, and for the type exclusions of Section~\ref{sec:types}. The same layer contains the arguments of Sections~\ref{sec:prelim}--\ref{sec:search}: Lemma~\ref{lem:gram} and its use in Proposition~\ref{prop:types}, Lemma~\ref{lem:S} for decompositions of length $22$, the implication (a)$\Rightarrow$(b) of Lemma~\ref{lem:lift} for the line $L_0$ and the deduction of Theorem~\ref{thm:l479} from its enumeration, the capacity table from the bounds of \cite{RR26}, the $23$ bounds of Table~\ref{tab:raise} that follow by monotonicity, and the deduction of Theorem~\ref{thm:main} from Proposition~\ref{prop:search}. Lemma~\ref{lem:S3} is not formalized; instead, the type $(19,0,3)$ is excluded by the computation behind Proposition~\ref{prop:types}, with $\Delta=1$ and the seven anchor orbits. The formal proof thus excludes six types by computation and two, $(18,3,1)$ and $(19,1,2)$, by Lemma~\ref{lem:S}.

The second layer evaluates the checkers on the certificates. We check the certificates as the search programs produce them and, apart from small samples, discard them once checked, since at about $13$ terabytes in total they are too large to keep. The kernel evaluates the checks for the type exclusions of Proposition~\ref{prop:types}, for the enumerations of $64$ classes of Table~\ref{tab:raise}, for the top levels of all other enumerations, and for the branch of the search in cases (I) and (I$'$) whose sets begin with a matrix of rank one, which contains no set with $k\le 13$. The other checks are evaluated as compiled code: those for the rest of the search at $r=22$ and, below their top levels, those for the enumerations of the point, of $23$ further classes and of Section~\ref{sec:l479}. The bounds of these $24$ classes therefore depend on compiled evaluations, and so, through monotonicity, do seven further bounds of Table~\ref{tab:raise}. In the enumeration of Section~\ref{sec:l479} the compiled checks leave out the $10$ profiles that admit decompositions of $T_{L_0}$, and for each of them a kernel proof shows that none of these decompositions lifts.

A compiled evaluation is a call of \texttt{native\_decide}. It runs the compiled checker on the data and adds to the environment an axiom stating that the checker returns \texttt{true} on these data. There is one such call for each of the $4{,}381{,}024$ batches of certificates and for each of $23$ further evaluations in the enumerations. The batches are combined through Lean's module system. Each batch is checked in its own module. Its data and its axiom are private to that module, which exports only a statement about the interface of the batch: the part of the search it covers and the subtrees it leaves to other batches. Further modules import these statements and prove in the kernel that the batches together cover the whole search and all enumerations. A theorem of the first layer turns this covering statement into Theorem~\ref{thm:main}. For this purpose the first layer is written in module form, and each of its more than two million constants that also occur in the non-module development has the same statement and the same axioms there. The covering statement and the final theorem are therefore in one environment, and \texttt{\#print axioms} for the final theorem lists the three standard axioms and the $4{,}381{,}047$ axioms added by these calls.

Each of these axioms is tied by name to a batch and to the checker that the batch evaluates. The checker is recorded with each batch, except for $187{,}578$ batches checked before such records were kept. For those it is determined by the kernel-checked statement that the batch exports, because only one soundness theorem in the build concludes that statement. For the $23$ evaluations outside the batches we checked directly that each axiom states that the intended checker returns \texttt{true}. The search programs and the programs that distribute the batches over the machines are not trusted: a wrong or missing batch makes a check or a covering proof fail.

\subsection{What the proof relies on}
The proof relies on the Lean kernel and, in addition, on the Lean compiler and runtime, the C compiler that builds the compiled checkers, and the arithmetic library they use, because \texttt{native\_decide} accepts the result of a compiled evaluation without a kernel check. For this reason Mathlib does not accept proofs that use \texttt{native\_decide}. This basis is weaker than that of the proofs by Yang et al.\ \cite{Yan26} and by Rudich and Rousseau \cite{RR26}, which use only the kernel. An independent kernel replay, as carried out in \cite{RR26}, re-checks the first layer and the covering proofs but not the evaluations. Checking a large finite computation by compiled evaluation of a checker proved sound is the method of the Coq proof of the four colour theorem \cite{Gon08} and of the enumeration of tame graphs in the formal proof of the Kepler conjecture \cite{Hal17}.

A second point concerns the library of compiled code. The compiled checkers are found by name, and Lean does not check that this library was built from the definitions to which the statement of a batch refers. With a library left over from an earlier build, a false evaluation would become an axiom that looks like any other. With a toy checker whose library had been compiled from an earlier version of its definition, we did obtain in this way a proof of \texttt{False} that Lean accepted. We therefore build each library into a directory named by a hash of its contents. Each batch records the SHA-256 hashes of the library and of the compiled checker modules that it loaded, and the script that generates the combining modules accepts a batch only if these hashes are those of the build used for the final theorem.

As a further check, we selected about one percent of the batches of each kind by a hash of their names and evaluated them again with Lean's interpreter, which runs the compiled code stored in the checker modules and does not use the library. Because the interpreter is $45$ to $200$ times slower than compiled code, we re-checked $382$ of the $744$ selected batches, mostly the smaller ones. Each of them reproduced the module of its batch byte for byte (Table~\ref{tab:checks}). This check does not depend on the library or on the C compiler, but it still depends on Lean's compiler, which produced the code that the interpreter runs. Section~\ref{sec:kernel} describes a second formal proof, in which every check is evaluated by the kernel.

\subsection{The computation}
Compiled, the checkers handle a node of the search at $r=22$ in about $1$--$2$~ms, a node of an enumeration in $0.004$--$0.03$~ms and a leaf in about $0.02$~ms. This is between one hundred and one thousand times faster than evaluation by the kernel. The computation used Lean v4.30.0 and about $478{,}000$ vCPU-hours of machine time.

\subsection{Further checks}
Table~\ref{tab:checks} lists checks that are not part of the proof. Their programs and outputs are in the accompanying data.

\begin{table}[htbp]
\caption{Independent checks. In the last two checks, a package is a group of Lean modules of the proof of Section~\ref{sec:kernel}, and the private part of a module consists of its certificates, which are private declarations (Section~\ref{subsec:otherchecks}).}\label{tab:checks}
\centering
\small
\begin{tabular}{@{}p{0.44\textwidth}p{0.53\textwidth}@{}}
\toprule
check & result\\
\midrule
The completion search on the first factors of a known decomposition of length $23$, with target $T$, and on its quotients by $12$ subspaces of $A$ & a completion is found in all $13$ cases\\
Six decompositions of length $23$ (listed in the data): the system \eqref{eq:excess} and its analogues for the second and third factors, $18$ cases in all & satisfied, with $\Delta\in\{1,2,3\}$\\
The $62{,}239$ decompositions of length $23$ of Section~\ref{subsec:upper}, each factor in the role of the first factor: the system \eqref{eq:excess} and Lemmas~\ref{lem:split}, \ref{lem:S} and~\ref{lem:S3}, $186{,}717$ cases & satisfied in all cases, with $\Delta$ from $0$ to $11$\\
The same $186{,}717$ cases against the bounds of \cite{RR26} and Table~\ref{tab:raise}: at most $23-\lb(S)$ first factors in each subspace $S$ of $A$ (Lemma~\ref{lem:cap}) & satisfied for all $8{,}283{,}458$ subspaces\\
Re-evaluation by Lean's interpreter, without the library, of $382$ of the batches selected by a hash of their names (Section~\ref{sec:verify}), namely $240$ of the $602$ selected for case (II), the one selected for each of cases (I) and (I$'$) and all $140$ selected for the enumerations, and of all $89$ parts of the search kept as a sample of about one in a thousand & identical modules in every case\\
Replay by Lean's checker of a random sample of $0.1\%$ of the private parts of the proof of Section~\ref{sec:kernel}, drawn from the enumerations for the point, the classes and the external leaves, which have $9.1$ million modules & none of the sampled parts fails\\
\leftskip=1em all $9{,}538$ sampled parts & $5{,}914$ replay, $3{,}258$ cannot be replayed because a private part they import was not kept, and $366$ lie in packages that the final proof does not use\\
\leftskip=1em the $9{,}165$ sampled parts in packages of the final proof & $5{,}912$ replay and $3{,}253$ cannot be replayed\\
Regeneration, with all private parts kept, of $24$ packages ($6{,}433$ modules) chosen at random in advance from the search at $r=22$ in the same proof, which has $5.0$ million modules: $20$ of the $18{,}110$ packages of case (II), the package of case (I) and $3$ of the $80$ packages of case (I$'$) & all $24$ packages export the statement of the original build\\
\leftskip=1em $6$ packages ($503$ modules) & a byte-identical replay of the original modules\\
\leftskip=1em $5$ packages ($576$ modules) & the same statement re-derived independently, with every module replayed by Lean's checker\\
\leftskip=1em the other $13$ packages ($5{,}354$ modules) & the same statement re-derived independently, with the combining modules replayed at the exported level\\
\leftskip=1em the $6{,}433$ regenerated modules & $6{,}390$ replay by Lean's checker and none fails; the other $43$, which combine the results of a package, replay in the kernel at the level of exported statements\\
\bottomrule
\end{tabular}
\end{table}

\subsection{Data and code}
The programs, the tables, the certificate formats, the Lean development and a list of all batches of the computation with their hashes are available at \url{https://github.com/tejasvi/rank333-f2} and archived at \url{https://doi.org/10.5281/zenodo.23236154}. The kernel-only proof of Section~\ref{sec:kernel}, with the generators and scripts that regenerate its modules, is at \url{https://github.com/tejasvi/rank333-f2-kernel} and archived at \url{https://doi.org/10.5281/zenodo.23304899}.

\section{Verification by the kernel}\label{sec:kernel}

The second formal proof evaluates every check in the Lean kernel. Its final theorem, \texttt{R22K.Final.rank\_ge\_23}, has the statement of Section~\ref{sec:verify}, and \texttt{\#print axioms} lists only \texttt{propext}, \texttt{Classical.choice} and \texttt{Quot.sound}. The same module proves the consequences stated in Section~\ref{sec:verify}, with the same axioms. The proof follows Sections~\ref{sec:prelim}--\ref{sec:tsolve} except in case (II), where the enumeration behind Theorem~\ref{thm:l479}, with $7.4\cdot 10^{11}$ nodes, is too large for the kernel. We replace it by Lemma~\ref{lem:W}, which has a short proof, and by Theorem~\ref{thm:planes}, whose enumerations are together ten times smaller.

\subsection{Two caps for case (II)}
The first lemma bounds the number of first factors in a subspace of dimension six.

\begin{lemma}\label{lem:W}
Let $W=\{a\in A: u^{\mathsf T}a=0\}$ or $W=\{a\in A: av=0\}$ with $u,v\in\F^3$ nonzero. In a decomposition of length at most $22$, at most $11$ first factors lie in $W$.
\end{lemma}

\begin{proof}
The transposition symmetry of Section~\ref{sec:prelim} exchanges the two kinds of subspaces, and the symmetries \eqref{eq:sym} act transitively on each kind. So we may assume that $W=W_0$ is the space of matrices with zero third row. Let $O$ be the set of terms with $a_t\notin W_0$ and $I$ the set of the others. Let $\pi\colon A\to A/W_0$ be the projection, which records the third row. Then
\[
\sum_{t\in O}\pi(a_t)\otimes b_t\otimes c_t=(\pi\otimes 1\otimes 1)(T)=\sum_{j,k}\pi(e_{3j})\otimes e_{jk}\otimes e_{k3}.
\]
Read as a linear map from $B^*$ to $(A/W_0)\otimes C$, the right side is injective with image $(A/W_0)\otimes C'$, where $C'$ is spanned by $e_{13},e_{23},e_{33}$. The left side factors through the map $N\colon B^*\to\F^{O}$, $\beta\mapsto(\beta(b_t))_t$, so $N$ is injective, $|O|\ge 9$, and the map $e_t\mapsto\pi(a_t)\otimes c_t$ sends the image of $N$, of dimension $9$, onto $(A/W_0)\otimes C'$. With $\rho\colon C\to C/C'$, the tensors $\pi(a_t)\otimes\rho(c_t)$ for $t\in O$ therefore span a space of dimension at most $s=|O|-9$. Each of them has rank at most one, and $\pi(a_t)\ne 0$, so all $\rho(c_t)$ lie in the span of at most $s$ of the vectors $\rho(c_t)$. Hence the third factors $c_t$ with $t\in O$ lie in a subspace $D\supseteq C'$ of $C$ of dimension at most $3+s$. Projecting the third factor to $C/D$ removes the terms of $O$, so the rank of $T$ modulo $D$ in the third factor is at most $|I|$. Since $T$ is invariant under the cyclic permutation of its three factors, this rank equals $R(D')$, where $D'\subseteq A$ has the same matrices as $D$. If $s\le 1$, then $D'$ is a subspace of dimension at most $4$ containing the space of matrices supported on the third column, and by \cite{RR26} each such subspace has $R(D')\ge 14$. This contradicts $R(D')\le|I|\le 22-|O|=13-s$. So $s\ge 2$ and $|I|\le 11$.
\end{proof}

Let
\[
E_{436}=\langle e_{11},\ e_{12}+e_{23},\ e_{22}+e_{31}\rangle\quad\text{and}\quad E_{422}=\langle e_{11},\ e_{13}+e_{22},\ e_{12}+e_{31}\rangle .
\]
In each of these planes the nonzero matrices have ranks $1,2,2,2,2,3,3$. The $G$-orbit of $E_{436}$ has $14{,}112$ planes and that of $E_{422}$ has $7{,}056$. Their stabilizers have orders $4$ and $8$, respectively. The table gives $R\ge 19$ for both, so Lemma~\ref{lem:cap} allows three first factors in such a plane.

\begin{theorem}[$\dagger$]\label{thm:planes}
In a decomposition of length $22$ with at least $14$ first factors of rank one, each plane in the $G$-orbit of $E_{436}$ or of $E_{422}$ contains at most two first factors.
\end{theorem}

The proof uses a lifting test and a symmetry reduction of the completion search.

\begin{lemma}\label{lem:liftplane}
Let $E$ be a plane and $T_E=\sum_{t=1}^{19}\bar a_t\otimes b_t\otimes c_t$ a decomposition with nonzero factors. For a linear form $\varphi$ on $A$ let $W_\varphi=T(\varphi)+\operatorname{span}\{b_tc_t^{\mathsf T}:1\le t\le 19\}$. Suppose that for each set $P$ of three nonzero elements of $E$ there is a form $\varphi$ such that every matrix in $W_\varphi$ has rank greater than the number of $x\in P$ with $\varphi(x)=1$. Then there are no $a_t\in\bar a_t$ and three terms with distinct first factors in $E$ that complete the $19$ terms to a decomposition of $T$.
\end{lemma}

\begin{proof}
Suppose such terms exist, with first factors in a set $P$. Contracting the decomposition of $T$ with $\varphi$ gives $T(\varphi)+\sum_t\varphi(a_t)b_tc_t^{\mathsf T}=\sum_{x\in P,\,\varphi(x)=1}u_xv_x^{\mathsf T}$ for some vectors $u_x,v_x$. The left side lies in $W_\varphi$, and the right side has rank at most the number of $x\in P$ with $\varphi(x)=1$, contrary to the choice of $\varphi$.
\end{proof}

Unlike Lemma~\ref{lem:lift}, Lemma~\ref{lem:liftplane} gives only a sufficient condition. Its hypothesis is certified by vectors $w_1,\dots,w_r$ and, for each nonzero $\lambda\in\F^r$, a vector $z_\lambda$ with $w_\lambda^{\mathsf T}T(\varphi)z_\lambda=1$ and $(w_\lambda^{\mathsf T}b_t)(c_t^{\mathsf T}z_\lambda)=0$ for all $t$, where $w_\lambda=\sum_i\lambda_iw_i$. Then $w_\lambda^{\mathsf T}Mz_\lambda=1$ for every $M\in W_\varphi$, so the rows $w_i^{\mathsf T}M$ are independent and $\rk M\ge r$. Where no single form suffices for a set $P$, the certificate fixes one element of $P$ and splits over the possible terms $uv^{\mathsf T}$ at that element, refuting each case in the same way.

The completion search on these leaves uses one more reduction. Two of the $19$ terms whose first factors lie in the same coset of $E$ can exchange their pairs $(b_t,c_t)$, and the result is again a decomposition of $T_E$ that lifts if and only if the original one does. Call a decomposition \emph{canonical} if within each such group of terms, apart from the group of the term on which the search branches first, the integers $512\,b_t+c_t$ do not decrease with $t$ (here $b_t$ and $c_t$ are read as $9$-bit integers).

\begin{lemma}\label{lem:eqf}
Each decomposition of $T_E$ becomes canonical after exchanges within the groups, and the exchanges do not change whether it lifts.
\end{lemma}

\begin{proof}
Let $\Phi=\sum_t t\,(512\,b_t+c_t)$. Exchanging an inverted pair $u<v$ in a group raises $\Phi$ by a positive amount, and $\Phi$ is bounded, so repeated exchanges end at a canonical decomposition. Both properties are invariant under exchanges because they depend only on the multiset of terms.
\end{proof}

The search therefore refutes only canonical decompositions. On four leaves for which the search was run both with and without this reduction, the number of complete assignments drops by exactly the order of the exchange group: $8$, $48$, $48$ and $96$.

\begin{proof}[Proof of Theorem~\ref{thm:planes}]
We follow the proof of Theorem~\ref{thm:l479}. Suppose that three first factors lie in $E=E_{436}$ or $E=E_{422}$. They are distinct by Corollary~\ref{cor:distinct}, and at most one of them has rank one, so at least $13$ of the other $19$ have rank one. The other terms project to a decomposition of $T_E$ of length $19$, whose coset profile on the $63$ nonzero cosets of $E$ satisfies \eqref{eq:coset} and $\sum_{x\notin C_1}n(x)\le 6$, where $C_1$ is the set of cosets that contain a matrix of rank one. We enumerate these profiles up to the stabilizer of $E$. Table~\ref{tab:planes} gives the size of the two enumerations. Inside the enumeration, the completion search with target $T_E$ refutes every leaf except the external ones, whose certificates are too long, whose search exceeds a budget, or which admit decompositions of $T_E$. A separate run of the completion search refutes most external leaves. For the others, decompositions of $T_E$ exist, and Lemma~\ref{lem:liftplane} shows that none of them lifts, with Lemma~\ref{lem:eqf} reducing the decompositions to be tested.
\end{proof}

\begin{table}[ht]
\caption{The enumerations behind Theorem~\ref{thm:planes}. External leaves are refuted outside the enumeration; for those with decompositions of $T_E$, the last column counts the complete assignments that the completion search tests with Lemma~\ref{lem:liftplane}.}\label{tab:planes}
\centering
\small
\begin{tabular}{@{}lrrrrr@{}}
\toprule
plane & nodes & leaves & external leaves & with decompositions & assignments\\
\midrule
$E_{436}$ & $60{,}864{,}120{,}396$ & $22{,}519{,}627{,}971$ & $392{,}525$ & $3{,}382$ & $343{,}609{,}305$\\
$E_{422}$ & $12{,}944{,}107{,}307$ & $4{,}503{,}216{,}713$ & $96{,}469$ & $27$ & $7{,}254$\\
\bottomrule
\end{tabular}
\end{table}

In case (II), the search of Section~\ref{sec:search} uses Lemma~\ref{lem:W} and Theorem~\ref{thm:planes} in place of Theorem~\ref{thm:l479}: in Definition~\ref{def:adm}(d), each plane in the orbits of $E_{436}$ and $E_{422}$ contains at most two elements of $F$, and each subspace of Lemma~\ref{lem:W} at most eleven.

\subsection{The other checks}\label{subsec:otherchecks}
The point, the $23$ classes of Table~\ref{tab:raise} whose enumerations the first proof evaluates as compiled code, and the search at $r=22$ are checked by the kernel in the same way as the other enumerations. For $21$ of the $23$ classes, none of their $7{,}446$ external leaves admits a decomposition of the quotient. For classes $303$ and $321$ the kernel proves the bound $R(S)\ge 19$ directly, and these proofs include the refutation of $3{,}587$ and $3{,}541$ external leaves, respectively.

The certificates are grouped into packages of Lean modules. The certificates of a module are private declarations, and the module exports a statement about the part of the proof it covers, as in Section~\ref{sec:verify}. The kernel checks every declaration when the module is built. The final theorem depends on about $151{,}000$ packages with $14.2$ million modules, built in about $530{,}000$ vCPU-hours. We read off the axioms of the final theorem in two independent ways: with \texttt{\#print axioms}, and as the union of the axioms that the kernel records for every module when it is built, private declarations included, over the $20{,}470$ modules that the final theorem uses. Both give \texttt{propext}, \texttt{Classical.choice} and \texttt{Quot.sound}. A search of all modules for \texttt{sorryAx}, \texttt{Lean.ofReduceBool}, \texttt{Lean.trustCompiler} and the auxiliary declarations of \texttt{native\_decide} finds none. We also checked that every module and every import matches its recorded hash, that no two modules share a name, and that all modules were checked by the same build of the kernel. The private parts of the modules were deleted after these checks, so a complete independent replay must first regenerate them with the published generators and rebuild the modules.

\section{Remarks}\label{sec:remarks}

The constraints of Section~\ref{sec:types} hold for all decompositions of length $22$, independently of the search. The Gram system \eqref{eq:excess} applies to decompositions of any length, and the computation behind Proposition~\ref{prop:types} can be applied to other types. It would be interesting to know whether the method says anything about the structure of decompositions of length $23$, of which there are many inequivalent ones over $\F$ \cite{CBH11,HKS21}.

The rank of $\mm$ over $\mathbb{F}_3$, over $\mathbb{Q}$ and over $\mathbb{C}$ remains open, as does its border rank over $\F$. Our method relies on enumerating finite sets of first factors over $\F$, and it does not extend to fields of characteristic zero.

\section*{Acknowledgment}
The author thanks Manuel Kauers for sharing the collection of schemes \cite{Kau}. This research received no specific funding.

\section*{Use of AI tools}
Generative AI tools contributed to this work as follows. They did not contribute to choosing the question, a long-standing open problem that the author chose to work on. They helped with the conceptual approach and made central technical contributions: they were used to find and prove the results, including the plan of the proof in Sections~\ref{sec:raise} to~\ref{sec:tsolve}, and to write the programs and both Lean formalizations. They were also used for writing and exposition (drafting and revising the text), for checking for errors (reviewing the proofs, the programs and the text), and for other supporting tasks (searching the literature and running the computations). The author takes full responsibility for the content of the paper, including the parts produced with these tools.

\end{document}